\RequirePackage[l2tabu,orthodox]{nag}

\documentclass{scrartcl}

\KOMAoption{fontsize}{10pt}

\usepackage{color,enumitem,graphicx, caption}
\usepackage[english]{babel}

\usepackage[normalem]{ulem}
\usepackage{amsfonts,amsmath,amssymb,latexsym,soul,amsthm, tikz-cd, braket, amsthm, mathtools, graphicx,array,color,colortbl,xcolor}
\usepackage{diagbox}
\usepackage{todonotes}

\usepackage{color,enumitem,graphicx, caption}
\usepackage{hyperref}
\hypersetup{colorlinks=true,urlcolor=blue,linkcolor=blue}

\DeclarePairedDelimiter{\abs}{\lvert}{\rvert}
\DeclarePairedDelimiter{\norm}{\lVert}{\rVert}

\usepackage[T1]{fontenc}
\usepackage{microtype}
\usepackage{mathrsfs}

\usepackage{fixmath}

\theoremstyle{plain}
\newtheorem{thm}{Theorem}[section]
\newtheorem{lemma}[thm]{Lemma}

\theoremstyle{remark}

\theoremstyle{definition}

\begin{document}
	\title{A note on nonlinear critical problems involving the Grushin Subelliptic Operator: bifurcation and multiplicity results}
	\author{Giovanni Molica Bisci \and Paolo Malanchini \and Simone Secchi}
	\maketitle
	\setcounter{secnumdepth}{2} 
	\setcounter{tocdepth}{2}
	
\begin{abstract}
We consider the  boundary value problem
\begin{equation*}
 \begin{cases}
 	-\Delta_\gamma u  = \lambda u + \left\vert u \right\vert^{2^*_\gamma-2}u ~&\text{in}~ \Omega\\
 	u = 0  ~&\text{on}~ \partial\Omega,
 \end{cases}
\end{equation*}
where $\Omega$ is an open bounded domain in $\mathbb{R}^N$, $N \geq 3$, while $\Delta_\gamma$ is the Grushin operator
\begin{displaymath}
\Delta_	\gamma u(z) = \Delta_x u(z) + \abs{x}^{2\gamma} \Delta_y u (z) \quad (\gamma\ge 0).
\end{displaymath} 
We prove a multiplicity and bifurcation result for this problem, extending the  results of Cerami, Fortunato and Struwe in \cite{cerami1984bifurcation} and of Fiscella, Molica Bisci and Servadei in \cite{fiscella2018multiplicity}.
\end{abstract}

\section{Introduction}

The degenerate differential operator 
\begin{displaymath}
	\frac{\partial^2}{\partial x^2} + x \frac{\partial^2}{\partial y^2} \quad \hbox{in $\mathbb{R}^2$}
\end{displaymath}
appeared for the first time one century ago in the seminal paper of Tricomi \cite{tricomi}. Several years later Baouendi \cite{baouendi} and Grushin \cite{grushin} generalized Tricomi's definition to what we now call the Grushin operator $\Delta_\gamma$, which was later studied extensively. We can roughly describe it as a weighted Laplace operator in which some of the independent variables are multiplied by a power-like function with a non-negative exponent $2\gamma$. When $\gamma=0$, the Grushin operator reduces formally to the basic Laplace operator. More precisely, the euclidean space $\mathbb{R}^N$ is split as $\mathbb{R}^m \times \mathbb{R}^\ell$, where $m \geq 1$, $\ell \geq 1$ and $m+\ell=N \geq 3$. A generic point $z \in \mathbb{R}^N$  has coordinates
\begin{displaymath}
	z = (x,y)=\left( x_1,\ldots,x_m,y_1,\ldots,y_\ell \right).
\end{displaymath}
For a differentiable function $u$ we define the Grushin gradient
\begin{displaymath}
\nabla_\gamma u (z) =  (\nabla_x u(z), \abs{x}^\gamma\nabla_y u(z))= (\partial_{x_1} u(z) , \dots, \partial_{x_m}u(z), \abs{x}^\gamma \partial_{y_{1}}u(z), \dots, \abs{x}^\gamma \partial_{y_{\ell}}u (z)).
\end{displaymath}
The Grushin operator $\Delta_\gamma$ is defined by\footnote{Some authors use the slightly different definition $\Delta_\gamma = \Delta_x + \left(1+\gamma \right)^2 \vert x \vert^{2\gamma}\Delta_y$.}
\begin{displaymath}
\Delta_	\gamma u(z) = \Delta_x u(z) + \abs{x}^{2\gamma} \Delta_y u (z),
\end{displaymath} 
where $\Delta_x$ and $\Delta_y$ are the Laplace operators in the variables $x$ and $y,$ respectively. A crucial property of the Grushin operator is that it is not uniformly elliptic in the space $\mathbb{R}^N$, since it is degenerate on the subspace $\{0\} \times \mathbb{R}^\ell$. For more technical reasons, see for instance the survey \cite{egorov}, the Grushin operator belongs to the class of \emph{subelliptic} operators, which halfway between elliptic and hyperbolic operators.

As remarked in \cite{kogojlanconelli}, if $\gamma$ is a non-negative integer, the operator $\Delta_\gamma$ falls into the class of H\"{o}rmander operators which are sum of squares of vector fields generating a Lie algebra of maximum rank at any point of the space. For a generic $\gamma \geq 0$, it is possible to embed $\Delta_\gamma$ into a family of operators of the form
\begin{displaymath}
	\Delta_\lambda = \sum_{i=1}^N \partial_{x_i} \left( \lambda_i^2 \partial_{x_i} \right),
\end{displaymath}
which have been studied from a geometrical viewpoint in \cite{franchilanconelli1982,franchilanconelli1983,franchilanconelli1984}. Here $\lambda_1,\ldots,\lambda_N$ are given functions which satisfy suitable conditions. The vector fields
\begin{displaymath}
	X_i = \lambda_i \partial_{x_i}, \quad i=1,\ldots,N
\end{displaymath}
define a weighted \emph{gradient} $\nabla_\lambda$.

\bigskip

In this paper we will be concerned with the semilinear elliptic equation
\begin{equation}
\label{problem:grushin}
 \begin{cases}
 	-\Delta_\gamma u  = \lambda u + \abs{u}^{2^*_\gamma-2}u &\text{in $\Omega$}\\
 	u = 0  ~&\text{on $\partial\Omega$},
 \end{cases}
\end{equation}
where $\Omega$ is a bounded domain of $\mathbb{R}^N$, $\lambda>0$ is a parameter,
\begin{gather*}
N_\gamma = m + (1+\gamma) \ell
\end{gather*}
is the homogeneous dimension associated to the decomposition $N=m+\ell$, and
\begin{gather*}
	2_\gamma^* = \frac{2N_\gamma}{N_\gamma -2}
\end{gather*}
is the Sobolev critical exponent associated to $N_\gamma$.
The case $\gamma=0$ corresponds (formally) to the celebrated Nonlinear Schr\"{o}dinger Equation with critical Sobolev exponent in a bounded domain $\Omega$ of $\mathbb{R}^N$, namely
\begin{equation}
\label{problem:1}
\begin{cases}
	-\Delta u = \lambda u + \vert u \vert^{2^*-2} u &\hbox{in $\Omega$} \\
	u=0 &\hbox{on $\partial \Omega$},
\end{cases}
\end{equation}
whose analysis was initiated by Brezis and Nirenberg in \cite{brezisnirenberg}. Here $N \geq 3$, $2^*=2N/(N-2)$ is the critical Sobolev exponent, and $\lambda >0$ is a positive parameter. By a careful comparison argument, the authors proved that for $N \geq 4$,
\begin{displaymath}
	\inf_{\substack{u \in H_0^1(\Omega) \\ \int_\Omega \vert u \vert^{2^*} =1}} \left\lbrace \int_\Omega \vert \nabla u \vert^2 - \lambda \int_\Omega \vert u \vert^2 \right\rbrace < \inf_{\substack{u\in H_0^1(\Omega) \\ \int_\Omega \vert u \vert^{2^*}=1}} \int_\Omega \vert \nabla u \vert^2 = S,
\end{displaymath}
where $S$ is the best constant for the continuous Sobolev embedding $H_0^1(\Omega) \subset L^{2^*}(\Omega)$. Calling $\lambda_1$ the first eigenvalue of the operator~$-\Delta$ defined on $H_0^1(\Omega) \subset L^2(\Omega)$, the existence of a positive solution to \eqref{problem:1} follows for every $\lambda \in (0,\lambda_1)$ if $N \geq 4$. The case $N=3$ is harder, and the authors proved the existence of a positive solution when $\Omega$ is a ball and $\lambda$ is sufficiently close to $\lambda_1$.

The paper \cite{brezisnirenberg} marked the beginning of an endless stream of efforts to extend the seminal results. In 1984, Cerami, Fortunato and Struwe proved in \cite{cerami1984bifurcation} a bifurcation result for problem \eqref{problem:1}. 
Similar multiplicity and bifurcation results have been proved for several classes of variational operators. For instance, in \cite{fiscella2018multiplicity} the authors considered a rather general family of fractional (i.e. nonlocal) operators which contains the usual fractional Laplace operator in $\mathbb{R}^N$.

\bigskip

Our aim is to extend the main result of \cite{cerami1984bifurcation} to problem \eqref{problem:grushin}.
With the notation introduced in Sections \ref{sec:preliminaries} and \ref{sec:eigenvalues}, our main result can be stated as follows. We denote by $\lambda_k$, $k=1,2,\ldots$, the Dirichlet eigenvalues of $-\Delta_\gamma$, see Section \ref{sec:eigenvalues}.
\begin{thm}
	\label{thm:main}
		Let $\Omega$ an open bounded subset of $\mathbb{R}^N$, $N\ge 3.$ Let $\lambda\in\mathbb{R}$ and $k \in \mathbb{N}$ be the smallest integer such that $\lambda_{k-1} \leq \lambda < \lambda_k$. Let $\lambda^* = \lambda_k$ and let $m\in\mathbb{N}$ be its multiplicity. Assume that
	\begin{equation}
		\label{eq:interval}
		\lambda^* - \frac{S_\gamma}{\vert \Omega \vert^{\frac{N_\gamma}{2}}}
		 < \lambda < \lambda^*,
	\end{equation}
	where $S_\gamma$ is the best critical Sobolev constant defined in (\ref{def:S}) and $\vert \Omega \vert$ denotes the volume of $\Omega$. Under these assumptions, problem (\ref{problem:grushin}) admits at least $m$ pairs of nontrivial solutions $(-u_{\lambda,i}, u_{\lambda,i})$, $i=1,\ldots,m$, such that
	\begin{displaymath}
		\norm{u_{\lambda,i}}_\gamma \to 0 ~\text{as}~\lambda\to \lambda^*
		\end{displaymath}
		for any $i=1,\ldots, m$. Furthermore, $u_{\lambda,i}\in L^\infty (\Omega)$ for each $i=1,\ldots,m$.
\end{thm}

While the Brezis-Nirenberg problem for the basic Laplace operator has been extensively studied over the past decades, its couterpart for the Grushin operator is somehow a recent topic in the literature. The best references which address it directly seem to be \cite{loiudice} and the very recent~\cite{alves2023}. Loiudice investigated in \cite{loiudice} the asymptotic behavior at infinity of positive solutions to the equation
\begin{displaymath}
	-\Delta_\gamma u = u^{2_\gamma^*-1}.
\end{displaymath}
In \cite{alves2023} the authors establish an existence result for the degenerate elliptic problem
\begin{displaymath}
\left\lbrace
	\begin{array}{ll}
		-\Delta_\gamma v = \lambda \vert v \vert^{q-2} v + \vert v \vert^{2_\gamma^*-2} v &\hbox{in $\Omega$} \\
		v =0 &\hbox{on $\partial \Omega$},
	\end{array}
	\right.
\end{displaymath}
where $\lambda>0$, $2 \leq q < 2_\gamma^*$. The domain $\Omega$ is required to satisfy the additional condition $\Omega \cap \lbrace x=0 \rbrace \neq \emptyset$. Our results do not overlap with theirs, even in the case $q=2$. The entire problem
\begin{equation}
\label{eq:entire_problem}
    \begin{cases}
        -\Delta_\gamma u = u^{2_\gamma^*-1} &\hbox{in $\mathbb{R}^N$}\\
        u>0
    \end{cases}
\end{equation}
has been studied in \cite{montimorbidelli,monti}, where uniqueness and symmetry properties of solutions have been proved. Of course \eqref{eq:entire_problem} is the critical point equation satisfied by the extremals of a Sobolev-type inequality for the Grushin operator.

\section{A functional setting for the Grushin operator} \label{sec:preliminaries}

For a bounded open set $\Omega\subset \mathbb{R}^N,$ we denote by 
$\mathscr{H}_\gamma $
the completion of $C^1_0(\Omega)$ with respect to the norm 
\begin{displaymath}
\norm{u}_\gamma = \left(\int_{\Omega}^{} \abs{\nabla_\gamma u} ^2 \, dz \right)^{1/2}.
\end{displaymath}
It turns out that $\mathscr{H}_\gamma$ is a Hilbert space with the inner product\footnote{For the sake of simplicity, we omit any scalar multiplication sign between the two vectors $\nabla_\gamma u$ and $\nabla_\gamma v$.}
\begin{displaymath}
\braket{u,v}_\gamma = \int_{\Omega}^{} \nabla_\gamma u \nabla_\gamma v\, \, dz.
\end{displaymath}
It is known (see \cite{alves2023berestycki}) that the continuous embedding
$\mathscr{H}_\gamma\hookrightarrow L^{2^*_\gamma} (\Omega)$ holds.
Moreover, if $\Omega$ is a bounded domain of $\mathbb{R}^N,$ then the embedding $\mathscr{H}_\gamma \hookrightarrow L^{p} (\Omega)$ is compact for every $p\in [1,2^*_\gamma),$ see \cite[Proposition 3.2]{kogojlanconelli}.

Let
\begin{equation}
	\label{def:S}
	S_\gamma = \inf\left\{ {\frac{\norm{u}_\gamma^2}{\norm{u}_{2^*_\gamma}^2} \mid u\in \mathscr{H}_\gamma\setminus\{0\}}\right\}
\end{equation}
denote the best constant for the embedding $\mathscr{H}_\gamma\hookrightarrow L^{2^*_\gamma}(\Omega)$.
A weak solution of problem (\ref{problem:grushin}) is by definition any function $u \in \mathscr{H}_\gamma$ such that
\begin{equation*}
	\int_{\Omega}^{}(\nabla_\gamma u \nabla_\gamma \varphi -\lambda u\varphi) \, dz =\int_{\Omega}^{} \abs{u}^{2^*_\gamma -2} u \varphi \, dz 
\end{equation*}
for all $\varphi \in \mathscr{H}_\gamma$.

The functional $I_\gamma \colon \mathscr{H}_\gamma \to \mathbb{R}$ defined by
\begin{displaymath}
I_\gamma(u) = \frac{1}{2} \int_{\Omega}^{}(\abs{\nabla_\gamma u}^2 - \lambda \abs{u}^2) \, dz -\frac{1}{2^*_\gamma}\int_{\Omega}^{} \abs{u}^{2^*_\gamma} \, dz
\end{displaymath}
is easily seen to be of class $C^1(\mathscr{H}_\gamma)$ with Fr\'{e}chet derivative
\begin{displaymath}
	DF(u) \colon v \mapsto \int_\Omega \left( \nabla_\gamma u \nabla_\gamma v - \lambda uv \right) \, dz - \int_\Omega \vert u \vert^{2_\gamma^* -2} u v \, dz,
\end{displaymath}

so that weak solutions to \eqref{problem:grushin} correspond to critical points of $I_\gamma$.

\section{Eigenvalues of the Grushin operator}
\label{sec:eigenvalues}

Since the proof of Theorem \ref{thm:main} is strictly related to the spectral theory associated to the operator $-\Delta_\gamma$, we collect in this section the basic spectral tools obtained  in \cite{xu2023nontrivial} for the Dirichlet eigenvalue problem
\begin{equation}
	\label{problem:eigenvalues}
	\begin{cases}
		-\Delta_\gamma u  = \lambda u ~&\text{in}~ \Omega\\
		u = 0  ~&\text{on}~ \partial\Omega.
	\end{cases}
\end{equation}
\begin{thm}
	\label{thm:autovalori}
	Let $\Omega$ be a bounded set of $\mathbb{R}^N$. Then the eigenvalues and eigenfunctions of $-\Delta_\gamma$ have the following properties:
	\begin{enumerate}[label=\arabic*.]
		\item The first eigenvalue $\lambda_1$ is given by
		\begin{displaymath}
		\lambda_1 = \min_{\substack{\norm{u}_2=1 \\ u\in \mathscr{H}_\gamma}} \int_{\Omega}^{} \abs{\nabla_\gamma u}^2 \, dz.
		\end{displaymath} 
		\item There exists a positive function $e_1\in \mathscr{H}_\gamma,$ which is an eigenfunction corresponding to $\lambda_1$ such that
		\begin{displaymath}
		\norm{e_1}_2 = 1, \quad\lambda_1 = \int_{\Omega}^{} \abs{\nabla_\gamma e_1}^2 \, dz.
		\end{displaymath}
		\item The first eigenvalue $\lambda_1$ is simple.
		\item The set of eigenvalues of (\ref{problem:eigenvalues}) consists of a sequence $\{\lambda_k\}_{k\in\mathbb{N}}$ such that
		\begin{displaymath}
		0<\lambda_1 < \lambda_2 \le \dots \le \lambda_k\le \lambda_{k+1} \le \dots
		\end{displaymath}
		and $\lambda_k \to+\infty$ as $k\to+\infty$.

		Moreover, for any $k\in\mathbb{N},$ the eigenvalues can be characterized as follows:
		\begin{displaymath}
		\lambda_{k+1} = \min_{\substack{\norm{u}_2 = 1 \\ u\in \mathbb{P}_{k+1}}} \int_{\Omega}^{} \abs{\nabla_\gamma u}^2 \, dz
		\end{displaymath} 
		where
		\begin{displaymath}
		\mathbb{P}_{k+1} = \left\{ u\in \mathscr{H}_\gamma \mid \braket{u,e_j} =0 \;\text{for } j=1,\dots, k \right\}.
		\end{displaymath}
		\item For any $k\in\mathbb{N},$ there exists a function $e_{k+1}\in\mathbb{P}_{k+1},$ which is an eigenfunction corresponding to $\lambda_{k+1},$ such that
		\begin{displaymath}
		\norm{e_{k+1}}_2 = 1, ~ \lambda_{k+1} = \int_{\Omega}^{} \abs{\nabla_\gamma e_{k+1}}^2 \, dz.
		\end{displaymath}
		\item The sequence $\{e_k\}_k$ of eigenfunctions associated to $\lambda_{k}$ is an orthonormal basis of $L^2(\Omega)$ and an orthogonal basis of $\mathscr{H}_\gamma.$
		\item Each eigenvalue $\lambda_{k}$ has finite multiplicity. More precisely, if $\lambda_k$ satisfies
		\begin{equation*}
			\lambda_{k-1}<\lambda_k = \dots = \lambda_{k+m} < \lambda_{k+m+1}
		\end{equation*}
		for some $m\in\mathbb{N}_0,$ then $\operatorname{span}\{e_k,\dots, e_{k+m}\}$ is the set of all the eigenfunctions corresponding to $\lambda_k$.
	\end{enumerate}
\end{thm}

\section{Proof of Theorem \ref{thm:main}}

We recall from \cite{bartolobencifortunato} the following abstract multiplicity result for critical points of an even functional.

\begin{thm}
	\label{thm:abstract} Let $H$ be a real Hilbert space with norm $\norm{\cdot}$ and suppose that $I\in C^1(H,\mathbb{R})$ is a functional on $H$ satisfying the following conditions:
	\begin{enumerate}[label=\arabic*.]
		\item $I(u) = I(-u)$ and $I(u) =0$;
		\item there exists a constant $\beta>0$ such that the Palais-Smale condition for $I$ holds in $(0,\beta);$
		\item there exist two closed subspaces $V$ and $W$ of $H$ and three positive constants $\rho$, $\delta$, $\beta'$ with $\delta <\beta'<\beta$, such that
		\begin{enumerate}[label=\alph*.]
			\item $I(u) \le \beta'$ for any $u\in W;$
			\item $I(u) \ge \delta$ for any $u\in V$ with $\norm{u} = \rho;$
			\item $\operatorname{codim}V<+\infty$ and $\dim W\ge \operatorname{codim}V$.
		\end{enumerate}
	\end{enumerate}
	Then there exists at least $\dim W - \operatorname{codim}V$ pairs of critical points of $I,$ with critical values belonging to the interval $[\delta,\beta']$.
\end{thm}

\bigskip

We are now ready to begin the proof of Theorem \ref{thm:main}, which we split in several steps.

\medskip

\begin{lemma}
\label{lem:4.2}
For any value
\begin{displaymath}
c< \frac{1}{N_\gamma} S_\gamma^{N_\gamma/2}
\end{displaymath} 
the functional $I_\gamma$ satisfies the Palais-Smale condition at level $c\in\mathbb{R}$.\footnote{This means that every sequence $\{u_j\}_{j\in\mathbb{N}}$ in $\mathscr{H}_\gamma$ such that $I_\gamma(u_j) \to c$ and $I_\gamma'(u_j) \to 0$ as $j \to +\infty$ admits a subsequence which converges strongly in $\mathscr{H}_\gamma$.}
\end{lemma}
%
\begin{proof}
Let $\{u_j\}_{j\in\mathbb{N}}$ be a sequence in $\mathscr{H}_\gamma$ which satisfies the Palais-Smale conditions.
First of all we claim that 
\begin{equation}
	\label{claim:1}
	\text{the sequence $\{u_j\}$ is bounded in $\mathscr{H}_\gamma.$}
\end{equation}
Indeed, for any $j\in\mathbb{N}$ there exists $k>0$ such that 
\begin{equation}
	\label{eq:1}
	\abs{I_\gamma(u_j)} \le k
\end{equation}
and
\begin{equation*}
	\left\vert
		\left\langle I'_\gamma (u_j), \tfrac{u_j}{\Vert u_j \Vert_\gamma} \right\rangle
	\right\vert \leq k
\end{equation*}
and so
\begin{equation}
	\label{eq:2}
	I_\gamma(u_j) - \frac{1}{2} \braket{I'_\gamma(u_j) , u_j} \le k(1+\norm{u_j}_\gamma).
\end{equation}
Furthermore,
\begin{displaymath}
I_\gamma(u_j) - \frac{1}{2} \braket{I'_\gamma(u_j) , u_j} = \left(\frac{1}{2}- \frac{1}{2^*_\gamma}\right) \norm{u_j}_{2_\gamma*}^{2_\gamma^*} = \frac{1}{N_\gamma} \norm{u_j}_{2_\gamma^*}^{2_\gamma^*} 
\end{displaymath}
so, thanks to (\ref{eq:2}), we get that for any $j\in\mathbb{N}$
\begin{equation}
	\label{eq:3}
	\norm{u_j}_{2_\gamma^*}^{2_\gamma^*}  \le \hat{k} (1+\norm{u_j}_\gamma)
\end{equation}
for a suitabile positive constant $\hat{k}.$ 
By H\"{o}lder's inequality, we get
\begin{displaymath}
\norm{u_j}_{2}^2 \le \abs{\Omega}^{{2}/{N_\gamma}} \norm{u_j}_{2_\gamma^*}^2 \le \hat{k}^{{2}/{2^*_\gamma}}\abs{\Omega}^{{2}/{N_\gamma}} (1+\norm{u_j}_\gamma)^{2/2^*_\gamma}
\end{displaymath}  
that is,
\begin{equation}
	\label{eq:4}
	\norm{u_j}_{2}^2 \le \tilde{k} (1+\norm{u_j}_\gamma),
\end{equation}                                                            
for a suitable $\tilde{k}>0$ independent of $j$. By (\ref{eq:1}), (\ref{eq:3}) and (\ref{eq:4}), we have that
\begin{align*} 
 k&\ge I_\gamma (u_j) = \frac{1}{2} \norm{u_j}_\gamma^2 - \frac{\lambda}{2} \norm{u_j}_{2}^2 - \frac{1}{2^*_\gamma}\norm{u_j}_{2^*_\gamma}^{2^*_\gamma} \\
 &\ge \frac{1}{2} \norm{u_j}_\gamma^2 - \overline{k} (1+\norm{u_j}_\gamma)
 \end{align*}
       with $\overline{k}>0$ independent of $j,$ so (\ref{claim:1}) is proved.
       
       Now, let us show that
       \begin{equation}
       	\label{claim:2}
       	\text{problem (\ref{problem:grushin}) admits a weak solution $u_\infty \in \mathscr{H}_\gamma$.}
       \end{equation}     
Since the sequence $\{u_j\}$ is bounded in $\mathscr{H}_\gamma$ and $\mathscr{H}_\gamma$ is a Hilbert space, then, up to a subsequence, still denoted by $\{u_j\},$ there exists $u_\infty\in \mathscr{H}_\gamma$ such that $u_j\rightharpoonup u_\infty$ in $\mathscr{H}_\gamma$, that is
\begin{equation}
	\label{eq:5}
	\int_{\Omega}^{} \nabla_\gamma u_j \nabla_\gamma v \, dz \to \int_{\Omega}^{} \nabla_\gamma u_\infty \nabla_\gamma v \, dz 
	\end{equation}
for any $v\in \mathscr{H}_\gamma$ as $j\to + \infty.$ Moreover, by (\ref{claim:1}), (\ref{eq:3}), the embedding properties of $\mathscr{H}_\gamma$ into the classical Lebesgue spaces and the fact that $L^{2^*_\gamma}(\Omega)$ is a reflexive space we have that, up to a subsequence
\begin{gather}
	\label{eq:6}
	u_j\rightharpoonup u_\infty ~\text{in} ~ L^{2^*_\gamma}(\Omega) \\
	\label{eq:6b}
	u_j\to u_\infty ~\text{in} ~ L^{2}(\Omega)
\end{gather}
and
\begin{equation}
	\label{eq:6c}
	u_j\to u_\infty ~\hbox{a.e in}~\Omega
\end{equation}
       as $j\to+\infty.$
As a consequence of (\ref{claim:1}) and (\ref{eq:3}), we have that $\norm{u_j}_{2^*_\gamma}$ is bounded uniformly in $j,$ hence the sequence $\{\abs{u_j}^{2^*_\gamma-2}u_j\}_j$ is bounded in $L^{{2^*_\gamma}/{2^*_\gamma -1}}(\Omega)$ uniformly in $j.$ Thus, from (\ref{eq:6}) we get
\begin{equation*}
	\abs{u_j}^{2^*_\gamma-2}u_j\rightharpoonup \abs{u_\infty}^{2^*_\gamma-2}u_\infty ~\text{in}~ L^{{2^*_\gamma}/{2^*_\gamma-1}}(\Omega)
\end{equation*}
       as $j\to+\infty,$ and so for all $\varphi\in L^{2^*_\gamma}(\Omega) = \left(L^{{2^*_\gamma}/{2^*_\gamma-1}}(\Omega)\right)^{'}$ we have that 
 \begin{equation*}
 	\int_{\Omega}^{} \abs{u_j}^{2^*_\gamma-2}u_j \varphi \, \, dz \to 
 	\int_{\Omega}^{} \abs{u_\infty}^{2^*_\gamma-2}u_\infty \varphi \, \, dz 
 \end{equation*}
       as $j\to+\infty$. In particular, for all $\varphi \in \mathscr{H}_\gamma$ we have that
 \begin{equation}
 	\label{eq:7}
 	 	\int_{\Omega}^{} \abs{u_j}^{2^*_\gamma-2}u_j \varphi \, \, dz \to 
 	\int_{\Omega}^{} \abs{u_\infty}^{2^*_\gamma-2}u_\infty \varphi \, \,dz
 \end{equation}
       as $j\to+\infty,$ since $\mathscr{H}_\gamma \subseteq L^{2^*_\gamma}(\Omega).$
       Recalling that $\lbrace u_j \rbrace_j$ is a Palais-Smale sequence, for any $\varphi\in \mathscr{H}_\gamma$ we have 
       \begin{displaymath}
       o(1)= \braket{I_\gamma'(u_j),\varphi} = \int_{\Omega}^{}\nabla_\gamma u_j \nabla_\gamma \varphi \, dz - \lambda \int_{\Omega}^{} u_j \varphi \, dz - \int_{\Omega}^{} \abs{u_j}^{2^*_\gamma-2}u_j\varphi \, dz
       \end{displaymath}
       as $j\to+\infty$.
       Letting $j \to +\infty$ and taking into account \eqref{eq:5}, \eqref{eq:6b} and \eqref{eq:7} we find
       \begin{displaymath}
       \int_{\Omega}^{}\nabla_\gamma u_\infty \nabla_\gamma \varphi \, \, dz - \lambda \int_{\Omega}^{} u_\infty \varphi \, \, dz - \int_{\Omega}^{} \abs{u_\infty}^{2^*_\gamma-2}u_\infty\varphi \, \, dz = 0.
       \end{displaymath}
       This shows that $u_\infty$ is a weak solution of problem (\ref{problem:grushin}), and (\ref{claim:2}) is proved.
Now we show that 
\begin{equation}
	\label{claim:3}
	u_j\to u_\infty \,\,\hbox{in}\,\, \mathscr{H}_\gamma.
\end{equation}
 Let $v_j = u_j  - u_\infty$. We claim that
 \begin{equation}
 	\label{claim:4}
 	\norm{v_j}_\gamma^2 = \norm{v_j}_{2^*_\gamma} ^ {2^*_\gamma} + o(1)
 \end{equation}
	Indeed, as a consequence of \eqref{eq:6}, \eqref{eq:6c}, \eqref{eq:7} and the Brezis-Lieb Lemma (see \cite[Theorem 1]{brezislieb}) we get
\begin{multline}
		\label{eq:9}
	\int_{\Omega} \left(\lvert u_j\rvert ^{2^*_\gamma -2} u_j - \lvert u_\infty \rvert ^{2^*_\gamma -2} u_\infty \right) \left(u_j - u_\infty\right) \,dz \\=
	\int_{\Omega} \abs{u_j}^{2^*_\gamma} \,dz - \int_{\Omega} \abs{u_\infty}^{2^*_\gamma} \, dz + o(1) = \int_\Omega \abs{u_j - u_\infty}^{2^*_\gamma} \, dz + o(1) \\=
	\norm{v_j}_{2^*_\gamma}^{2^*_\gamma} + o(1)
\end{multline}
By \eqref{claim:2} and the properties of the functional $I_\gamma$  we know that
\begin{displaymath}
	o(1) = \braket{I'_\gamma(u_j), v_j} = \braket{I'_\gamma (u_j) - I'_\gamma (u_\infty), v_j},
\end{displaymath}
that is
\begin{multline*}
	 o(1) = \int_{\Omega} \left(\lvert \nabla_\gamma u_j\rvert ^2 - 2\nabla_\gamma u_j\nabla_\gamma u_\infty + \lvert \nabla_\gamma u_\infty \rvert^2\right) \,dz - \lambda \int_\Omega \left(u_j^2 -2u_ju_\infty + u_\infty^2\right)\,dz \\
	 -\int_{\Omega} \left(\left(\lvert u_j\rvert ^{2^*_\gamma -2}u_j - \lvert u_\infty \rvert^{2^*_\gamma-2} u_\infty\right) \left(u_j-u_\infty\right)\right)\,dz \\=
	 \norm{v_j}_\gamma^2 -\lambda \norm{v_j}_2^2 -	\int_{\Omega} \left(\lvert u_j\rvert ^{2^*_\gamma -2} u_j - \lvert u_\infty \rvert ^{2^*_\gamma -2} u_\infty \right) \left(u_j - u_\infty\right) \,dz
	\end{multline*}
	Now, by \eqref{eq:6b} and \eqref{eq:9}
		\begin{displaymath}
		\norm{v_j}_\gamma^2 -\lambda \norm{v_j}_2^2 -	\int_{\Omega} \left(\lvert u_j\rvert ^{2^*_\gamma -2} u_j - \lvert u_\infty \rvert ^{2^*_\gamma -2} u_\infty \right) \left(u_j - u_\infty\right) \,dz=
\norm{v_j}_\gamma^2 - \norm{v_j}_{2^*_\gamma}^{2^*_\gamma} + o(1),
	\end{displaymath}
so \eqref{claim:4} is proved. Since $\braket{I'_\gamma (u_j), u_j} \to 0$
we have
\begin{displaymath}
\norm{u_j}_{2^*_\gamma}^{2^*_\gamma} = \int_{\Omega}^{} \left(\abs{\nabla_\gamma (u_j)}^2 - \lambda \abs{u_j}^2\right) \, dz + o(1).
\end{displaymath}
Inserting this equality into the expression for $I_\gamma(u_j)$ we obtain
\begin{multline}
		\label{eq:10}
			I_\gamma(u_j) = \left(\frac{1}{2}- \frac{1}{2^*_\gamma}\right) \int_{\Omega}^{} \left(\abs{\nabla_\gamma u_j}^2 - \lambda \abs{u_j}^2\right) \, dz	+ o(1)  \\
		=\frac{1}{N_\gamma}  \int_{\Omega}^{} \left(\abs{\nabla_\gamma u_\infty}^2- \lambda \abs{u_\infty}^2 \right)\, dz  + \frac{1}{N_\gamma} \int_{\Omega}^{} \abs{\nabla_\gamma v_j}^2 \, dz + o(1)	.
\end{multline}
Moreover, since $u_\infty$ is a weak solution of (\ref{problem:grushin})
\begin{displaymath}
\frac{1}{2} \int_{\Omega}^{} \left(\abs{\nabla_\gamma u_\infty}^2 - \lambda \abs{u_\infty}^2\right) \, dz -\frac{1}{2^*_\gamma}\int_{\Omega}^{}\abs{u_\infty}^{2^*_\gamma} \, dz= 0,
\end{displaymath}
in particular
\begin{equation}
	\label{eq:positive:int}
	\int_{\Omega}^{} \left(\abs{\nabla_\gamma u_\infty}^2 -\lambda \abs{u_\infty}^2\right) \, dz \ge 0.
	\end{equation}
From (\ref{eq:10}) and (\ref{eq:positive:int}) we now infer
\begin{displaymath}
\norm{v_j}_\gamma^2 \le N_\gamma I_\gamma (u_j) + o(1).
\end{displaymath}
Since $\lbrace v_j \rbrace_j$ is a Palais-Smale sequence, for $j$ sufficiently large we obtain
\begin{equation}
	\label{eq:11}
	\norm{v_j}_\gamma^2 \le c_2 < S_\gamma^{N_\gamma/2}.
\end{equation}
Now, by the definition of the constant $S_\gamma$
\begin{equation*}
	\begin{split}
		\norm{v_j}_{2^*_\gamma}^2 \le \frac{\norm{v_j}_\gamma^2}{S_\gamma}
	\end{split}
\end{equation*}
that is
\begin{equation*}
	\begin{split}
		\norm{v_j}_{2^*_\gamma}^{2^*_\gamma} \le \frac{\norm{v_j}_\gamma^{2^*_\gamma}}{S_\gamma^{{N_\gamma}/{N_\gamma -2}}} = \norm{v_j}_\gamma^{2^*_\gamma} S_\gamma^{{2^*_\gamma}/{2}}
	\end{split}
\end{equation*}
and by (\ref{claim:4})
\begin{displaymath}
\norm{v_j}_\gamma^2 \le \norm{v_j}_\gamma^{2^*_\gamma} S_\gamma^{-{2^*_\gamma}/{2}} + o(1)
\end{displaymath}                        
that is
\begin{displaymath}
\norm{v_j}_\gamma^2 (S_\gamma^{{2^*_\gamma}/{2}}- \norm{v_j}_\gamma^{2^*_\gamma-2} ) \le o(1). \
\end{displaymath}
Taking account of (\ref{eq:11}) this implies that $v_j\to 0$ strongly in $\mathscr{H}_\gamma,$ concluding the proof of \eqref{claim:3}.
\end{proof}

\begin{lemma}
\label{lem:4.3}
	The functional $I_\lambda$ satisfies the assumption 3 of Theorem \ref{thm:abstract}.
\end{lemma}
\begin{proof}
Let $\lambda^*$ be as in the statement of Theorem
\ref{thm:main}. Then
\begin{displaymath}
\lambda^* = \lambda_k
\end{displaymath}
for some $k\in\mathbb{N}.$ Since $\lambda^*$ has multiplicity
$m\in\mathbb{N}$ by assumption, we have that
\begin{align*}
	\lambda^* =&~ \lambda_1 <\lambda_2 &\text{if} ~ k=1\\
	\lambda_{k-1} <\lambda^* =&~ \lambda_k = \dots = \lambda_{k+m-1}<\lambda_{k+m}  &\text{if} ~ k\ge 2. 
\end{align*}
Before proceeding with the proof of Theorem \ref{thm:main} we
claim that, under condition (\ref{eq:interval}),
\begin{equation}
	\label{eq:positive}
	\lambda>0.
\end{equation}
Indeed, by definition of $\lambda^*$ and taking into account Theorem
\ref{thm:autovalori} it is easily seen that
\begin{equation}
	\label{eq:lambda}
	\lambda\ge \lambda_1.
\end{equation}
Now, by the variational characterization of the first eigenvalue
$\lambda_1$ and by H\"{o}lder's inequality:
\begin{displaymath}
\norm{u}_2^2 \le \abs{\Omega}^{2/N_\gamma}\norm{u}_{2^*_\gamma}^2
\end{displaymath}
we obtain
\begin{displaymath}
\lambda_1 \ge S_\gamma\abs{\Omega}^{-2/N_\gamma}
\end{displaymath}
and \eqref{eq:positive} follows from \eqref{eq:lambda}.

With the notations of the abstract result stated in Theorem \ref{thm:autovalori}, we set
\begin{displaymath}
W = \operatorname{span} \{e_1, \dots, e_{k+m-1}\}
\end{displaymath}
and 
\begin{equation*}
	V=
	\begin{cases}
		\mathscr{H}_\gamma\quad &\text{if} ~ k=1 \\
		\{u\in \mathscr{H}_\gamma: \braket{u,e_j}_\gamma = 0 ~\forall j=1,\dots, k-1\} \quad&\text{if} ~k\ge 2.
		\end{cases}
\end{equation*}
Note that both $W$ and $V$ are closed subsets of $\mathscr{H}_\gamma$ and
\begin{equation*}
	\dim W = k +m -1 \quad \operatorname{codim} V = k-1
\end{equation*}
so that condition (3c) of Theorem \ref{thm:abstract} is satisfied.
Let us show that the functional $I_\gamma$ verifies the assumptions $(3a)$ of Theorem \ref{thm:abstract}. For this, let $u\in W$. Then,
\begin{displaymath}
u = \sum_{i=1}^{k+m-1} u_i e_i, \quad u_i\in\mathbb{R}, \; i=1\dots k+m-1.
\end{displaymath}
Since $\{e_i\}_i$ is an orthonormal basis of $L^2(\Omega)$ and an orthogonal basis of $\mathscr{H}_\gamma$, \cite[Theorem 1]{xu2023nontrivial}, we get
\begin{displaymath}
\norm{u}_\gamma^2 = \sum_{i=1}^{k+m-1} u_i^2 \norm{e_i}_\gamma^2 = \sum_{i=1}^{m+k-1} \lambda_i u_i^2 \le \lambda_k \sum_{i=1}^{m+k-1} u_i^2 = \lambda^* \norm{u}_2^2
\end{displaymath}
so that, by this and the H\"{o}lder's inequality, we have
\begin{equation}
	\label{eq:12}
		I_\gamma \le \frac{1}{2} (\lambda^*-\lambda) \norm{u}_2^2 -\frac{1}{2^*_\gamma}\norm{u}_{2^*_\gamma}^{2^*_\gamma} \le 
		\frac{1}{2} (\lambda^*-\lambda) \abs{\Omega} ^{{2}/{N_\gamma}}\norm{u}_{2^*_\gamma}^2 -\frac{1}{2^*_\gamma}\norm{u}_{2^*_\gamma}^{2^*_\gamma}
  \end{equation}
Now, for $t\ge 0$ we define the function
\begin{displaymath}
g(t) = \frac{1}{2} (\lambda^*-\lambda) \abs{\Omega}^{2/N_\gamma} t^2- \frac{1}{2^*_\gamma} t^{2^*_\gamma}.
\end{displaymath}
For every $t>0$ we have
\begin{displaymath}
g'(t) = (\lambda^*-\lambda) \abs{\Omega} ^{2/N_\gamma}t - t^{2^*_\gamma-1}
\end{displaymath}
so that $g'(t) \ge 0$ if and only if
\begin{displaymath}
t\le \overline{t} = \left[(\lambda^*-\lambda) \abs{\Omega}^{2/N_\gamma}\right]^{1/(2^*_\gamma -2)}.
\end{displaymath}
As a consequence of this, $\overline{t}$ is a maximum point for $g$ and so for any $t\ge 0$
\begin{equation}
	\label{eq:13}
	g(t) \le \max_{t\ge 0} g(t) = g(\overline{t}) = \frac{1}{N_\gamma} (\lambda^*-\lambda) ^{N_\gamma/2} \abs{\Omega}
\end{equation}
By (\ref{eq:12}) and (\ref{eq:13}) we get
\begin{equation*}
	\sup_{u\in W} I_\gamma (u) \le \max_{t\ge 0} g(t) = \frac{1}{N_\gamma} (\lambda^*-\lambda) ^{N_\gamma/2} \abs{\Omega} >0 
\end{equation*}
so condition $(3a)$ is satisfied with $\beta' = \frac{1}{N_\gamma} (\lambda^*-\lambda) ^{N_\gamma/2} \abs{\Omega}.$

Now let us prove that $I_\gamma$ satisfies condition $(3b)$. For every $u \in V$ there holds
\begin{equation*}
	\norm{u}_\gamma^2 \ge \lambda^* \norm{u}_2^2
\end{equation*}
Indeed,  the assertion is trivial if $u\equiv 0$, while if $u\in V\setminus \{0\}$ it follows from the variational characterization of $\lambda^* = \lambda_k$ given by
\begin{displaymath}
\lambda_k = \min_{u\in V\setminus \{0\}} \frac{\norm{u}_\gamma^2}{\norm{u}_2^2}
\end{displaymath}
as proved in \cite[Theorem 1]{xu2023nontrivial}. 
Thus, by the definition of $S_\gamma$ and taking into account that $\lambda>0,$ it follows that
\begin{equation*}
	I_\gamma (u) \ge \frac{1}{2} \left( 1- \frac{\lambda}{\lambda^*}\right) \norm{u}_\gamma^2 - \frac{1}{2^*_\gamma S_\gamma^{2^*_\gamma/2}} \norm{u}_\gamma^{2^*_\gamma} = \norm{u}_\gamma^2 \left( \frac{1}{2} \left( 1- \frac{\lambda}{\lambda^*}\right) - \frac{1}{2^*_\gamma S_\gamma^{2^*_\gamma/2}} \norm{u}_\gamma^{2^*_\gamma -2}\right)
\end{equation*}
Now, let $u\in V$ be such that $\norm{u}_\gamma = \rho>0.$ Since $2^*_\gamma>2,$ we can choose $\rho$ sufficiently small, say $\rho\le\overline{\rho}$ with $\overline{\rho}> 0,$ so that
\begin{equation*}
	\frac{1}{2} \left( 1- \frac{\lambda}{\lambda^*}\right) - \frac{1}{2^*_\gamma S_\gamma^{2^*_\gamma/2}} \norm{u}_\gamma^{2^*_\gamma -2}>0
\end{equation*}  
 and
 \begin{displaymath}
 \rho^2 \left( \frac{1}{2} \left( 1- \frac{\lambda}{\lambda^*}\right) - \frac{1}{2^*_\gamma S_\gamma^{2^*_\gamma/2}} \rho^{2^*_\gamma -2}\right) <  \rho^2 \left( \frac{1}{2} \left( 1- \frac{\lambda}{\lambda^*}\right)\right) < \beta' = \frac{1}{N_\gamma} (\lambda^*-\lambda) ^{N_\gamma/2} \abs{\Omega}
 \end{displaymath}
 \end{proof}
Now we can conclude the proof of Theorem \ref{thm:main}. It is clear that $I_\lambda$ satisfies assumption 1 of Theorem \ref{thm:abstract}. From Lemma \ref{lem:4.2} it easily follows that $I_\gamma$ satisfies condition (2) of Theorem \ref{thm:abstract} with 
\begin{displaymath}
\beta = \frac{1}{N_\gamma} S_\gamma^{N_\gamma/2}.
\end{displaymath}
Also, by Lemma \ref{lem:4.3} we get that $I_\gamma$ verifies condition (3) with
\begin{align*}
	\beta' &= \frac{1}{N_\gamma} (\lambda^*-\lambda) ^{N_\gamma/2} \abs{\Omega}\\
	\rho &= \overline{\rho} \\
	\delta &=  \rho^2 \left( \frac{1}{2} \left( 1- \frac{\lambda}{\lambda^*}\right) - \frac{1}{2^*_\gamma S_\gamma^{2^*_\gamma/2}} \rho^{2^*_\gamma -2}\right)
\end{align*}
Observe that $0<\delta<\beta'<\beta$.
By Theorem \ref{thm:abstract}, $I_\gamma$ has $m$ pairs $\{-u_{\lambda,u}, u_{\lambda,i}\}$ of critical points whose critical values $I(\pm u_{\lambda,i})$ are such that 
\begin{equation}
	\label{eq:14}
	0<\delta\le I_\gamma(\pm u_{\lambda,i}) \le \beta'
\end{equation}
for any $i=1,\dots, m.$
Since $I_\gamma(0) =0$ and \eqref{eq:14} holds true, it is easy to see that these critical points are all different from the trivial function. Hence problem (\ref{problem:grushin}) admits $m$ pairs of nontrivial weak solutions.

Now, fix $i\in\{1,\dots, m\}.$ By (\ref{eq:14}) we obtain
\begin{displaymath}
\beta' \ge I_\gamma(u_{\lambda,i}) = I_\gamma(u_{\lambda,i}) -\frac{1}{2} \braket{I'_\gamma(u_{\lambda,i}, u_{\lambda,i})} = \left(\frac{1}{2} -\frac{1}{2^*_\gamma}\right) \norm{u_{\lambda,i}}_{2^*_\gamma}^{2^*_\gamma} = \frac{1}{N_\gamma}\norm{u_{\lambda,i}}_{2^*_\gamma}^{2^*_\gamma}
\end{displaymath}
so that, passing to the limit as $\lambda\to \lambda^*$, it follows that (note that $\beta'\to 0$ as $\lambda\to \lambda^*$)
\begin{equation}
	\label{eq:15}
	\norm{u_{\lambda,i}}_{2^*_\gamma}^{2^*_\gamma} \to 0.
\end{equation}
Then, since $L^{2^*_\gamma} (\Omega) \hookrightarrow L^2(\Omega)$ continuously (being $\Omega$ bounded), we also get
\begin{equation}
	\label{eq:16}
	\norm{u_{\lambda,i}}_2^2 \to 0 \quad \text{as} \quad \lambda\to \lambda^*.
\end{equation}
Arguing as above, we have
\begin{displaymath}
\beta' \ge I_\gamma(u_{\lambda,i}) = \frac{1}{2} \norm{u_{\lambda,i}}_\gamma^2 -\frac{\lambda}{2} \norm{u_{\lambda,i}}_2^2 - \frac{1}{2^*_\gamma} \norm{u_{\lambda,i}}_{2^*_\gamma}^{2^*_\gamma}
\end{displaymath}
which, combined with (\ref{eq:15}) and (\ref{eq:16}) gives
\begin{displaymath}
\norm{u_{\lambda,i}}_\gamma \to 0 \quad \text{as} \quad \lambda\to\lambda^*.
\end{displaymath}
To conclude, we apply \cite[Theorem 3.1 and Remark 3.2]{loiudice} to the operator $V=\lambda + \abs{u}^{2^*_\gamma -2}\in L^{N_\gamma/2}(\Omega)$ and obtain that the solutions $u_{\lambda,i}$ are bounded.
This concludes the proof of the Theorem \ref{thm:main}.

\section*{Acknowledgements}

The three authors are members of the GNAMPA section of INdAM, Istituto Nazionale di Alta Matematica.

\bibliographystyle{plain}
\bibliography{Grushin_operator}

\bigskip

Paolo Malanchini, Dipartimento di Matematica e Applicazioni, Universit\`a degli Studi di Milano Bicocca, via R. Cozzi 55, I-20125 Milano, Italy.

\medskip

Giovanni Molica Bisci, Dipartimento di Scienze Pure e Applicate, Universit\`a di Urbino Carlo Bo, Piazza della Repubblica 13, Urbino, Italy.

\medskip

Simone Secchi, Dipartimento di Matematica e Applicazioni, Universit\`a degli Studi di Milano Bicocca, via R. Cozzi 55, I-20125 Milano, Italy.
\end{document}